\documentclass[oneside]{amsart}
\usepackage[utf8]{inputenc}
\usepackage{amstext}
\usepackage{amsthm}
\usepackage{amssymb}
\usepackage{graphicx}
\usepackage{esint}

\makeatletter
\numberwithin{equation}{section}
\numberwithin{figure}{section}
\theoremstyle{plain}
\newtheorem{thm}{\protect\theoremname}
\theoremstyle{plain}
\newtheorem{lem}[thm]{\protect\lemmaname}

\global\long\def\Im{\operatorname{Im}}

\global\long\def\Arg{\operatorname{Arg}}

\makeatother

\providecommand{\lemmaname}{Lemma}
\providecommand{\theoremname}{Theorem}

\begin{document}
\title[Arithmetic sequence of exponents]{Entire functions with an arithmetic sequence of exponents}
\author{Dallas Ruth \and Khang Tran}
\email{khangt@mail.fresnostate.edu \and dallasruth@mail.fresnostate.edu}
\address{Department of Mathematics, California State University, Fresno.\\
5245 North Backer Avenue M/S PB108 Fresno, CA 93740}
\subjclass[2000]{30C15; 26C10; 11C08}
\begin{abstract}
For a given entire function $f(z)=\sum_{j=0}^{\infty}a_{j}z^{j}$,
we study the zero distribution of $f_{r}(z)=\sum_{j\equiv r\pmod m}a_{j}z^{j}$
where $m\in\mathbb{N}$ and $0\le r<m$. We find conditions under
which the zeros of $f_{r}(z)$ lie on $m$ radial rays defined by
$\Im z^{m}=0$ and $\Re z^{m}\le0$. 
\end{abstract}

\maketitle

\section{Introduction}

The study of zero distribution of polynomials has a long history which
spans at least from a theorem of René Descartes on the rule of signs
to the modern theory of linear operators preserving zeros of polynomials
on a circular domain developed by J. Borcea and P. Br\"and\'en (\cite{bb}).
On one end, the Descartes' rule of signs, whose proof can be found
in \cite{wang}, asserts that the number of positive zeros of a real
polynomial is equal to or less than the number of sign changes of
its coefficients by an even number. This rule will play an important
role in our paper later in locating the final zero of our polynomials.
On the other end, a special linear operator has a connection with
the famous Hermite--Biehler theorem (Theorem 6.3.4 of \cite{rs})
via the linear transformation which transforms a polynomial 
\[
P(z)=\sum_{j=0}^{n}a_{j}z^{j}
\]
to the polynomial 
\[
P_{r}(z)=\sum_{\substack{j\equiv r\pmod m\\
0\le j\le n
}
}a_{j}z^{j}
\]
for $m\in\mathbb{N}$ and $0\le r<m$. If we define the sequence 
\[
\gamma_{j}=\begin{cases}
1 & \text{ if }j\equiv r\pmod m\\
0 & \text{ else}
\end{cases},
\]
then 
\[
P_{r}(z)=\sum_{j=0}^{n}\gamma_{j}a_{j}z^{j}.
\]
By P\'olya and Schur's characterization \cite{ps}, the sequence
$\left\{ \gamma_{j}\right\} _{j=0}^{\infty}$ is not a multiplier
sequence, a sequence in which if the zeros of $\sum_{j=0}^{n}a_{j}z^{j}$
are real, then so are those of $\sum_{j=0}^{n}\gamma_{j}a_{j}z^{j}$.
However, Hermite--Biehler theorem asserts that if $m=2$ and the
zeros of $P(z)$ lies on the open left-half plane, then the zeros
of $P_{r}(z)$, $0\le r\le1$, are purely imaginary. It is natural
to ask for similar results for other values of $m$. Theorems \ref{thm:zerossectorpoly}
and \ref{thm:zerohalfplanepoly} of this paper give an answer to this
question for the case $m=3$ and $4$. Finally Theorems \ref{thm:zerossectorentire}
and \ref{thm:zeroshalfplaneentire} extend these results to special
classes of entire functions of orders $1$ and $2$ respectively.
These classes of entire functions are different from the Laguerre-P\'olya
class (\cite[Definitions 2 and 3]{cf}) in the sense that the zeros
of these functions lie either on a sector of the plane or a half plane
(instead of on the real line). 

In the whole paper, for a complex number $z$, we use the principle
angle which is restricted by $-\pi<\Arg z\le\pi$. With this convention,
we provide formal statements of our theorems in this paper.
\begin{thm}
\label{thm:zerossectorpoly}Suppose $P(z)=\sum_{i=0}^{n}a_{i}z^{i}\in\mathbb{R}[z]$
is a real polynomial whose zeros lie on the sector of the left half
plane defined by the complex argument inequalities $2\pi/3<\left|\Arg z\right|\le\pi$.
Then for any $0\le r<3$, the zeros of the polynomial 
\[
P_{r}(z)=\sum_{\substack{j\equiv r\pmod3\\
0\le j\le n
}
}a_{j}z^{j}
\]
lie on the $3$ radial rays: $\Im z^{3}=0$ and $\Re z^{3}\le0$. 
\end{thm}

\,
\begin{thm}
\label{thm:zerohalfplanepoly}Suppose $P(z)=\sum_{i=0}^{n}a_{i}z^{i}\in\mathbb{R}[z]$
is a real polynomial whose zeros lie on the open left half plane $\Re z<0$.
Then for any $0\le r<4$, the zeros of the polynomial 
\[
P_{r}(z)=\sum_{\substack{j\equiv r\pmod4\\
0\le j\le n
}
}a_{j}z^{j}
\]
lie on the $4$ radial rays: $\Im z^{4}=0$ and $\Re z^{4}\le0$. 
\end{thm}

For an example of Theorems \ref{thm:zerossectorpoly} and \ref{thm:zerohalfplanepoly},
we consider
\begin{align*}
P(z) & =(z+1+i)^{2}(z+1-i)^{2}(z+1)^{3}(z+4+2i)(z+4-2i)\\
 & =z^{9}+15z^{8}+99z^{7}+369z^{6}+876z^{5}+1392z^{4}+1492z^{3}+1044z^{2}+432z+80.
\end{align*}
In the case $(m,r)=(3,0)$ or $(m,r)=(4,0)$, the polynomial $P_{r}(z)$
is 
\[
z^{9}+369z^{6}+1492z^{3}+80
\]
or 
\[
80+1392z^{4}+15z^{8}
\]
respectively. The zeros of $P_{r}(z)$ and the rays containing these
zeros are plotted in Figure \ref{fig:zerospoly}. 

\begin{figure}
\begin{centering}
\includegraphics[scale=0.4]{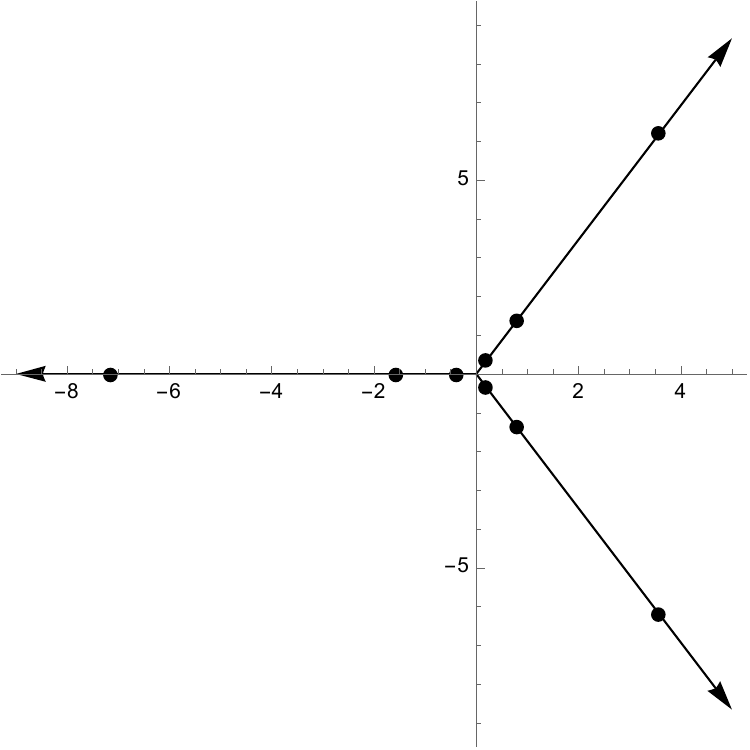}$\qquad$\includegraphics[scale=0.4]{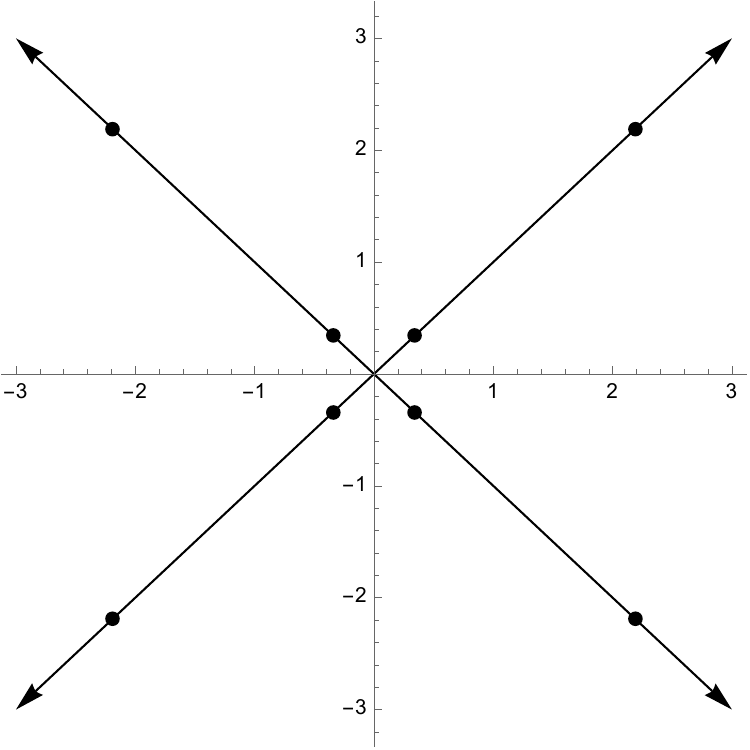}
\par\end{centering}
\caption{\label{fig:zerospoly}Zeros of $P_{r}(z)$ when $(m,r)=(3,0)$ (left)
and $(m,r)=(4,0)$ (right)}
\end{figure}

As we see in the statements of Theorems \ref{thm:zerossectorpoly}
and \ref{thm:zerohalfplanepoly}, the condition of the region (whether
a sector of the plane or a half-plane) containing zeros of $P(z)$
depends of the parity of $m$. Indeed, the conclusion of Theorem \ref{thm:zerossectorpoly}
will not hold if we replace this region, $2\pi/3<\left|\Arg z\right|\le\pi$,
by the half-plane $\Re z<0$. For example, the zeros of 
\[
P(z)=((z+1)^{2}+10^{2})^{3}=z^{6}+6z^{5}+315z^{4}+1220z^{3}+31815z^{2}+61206z+1030301
\]
are $-1\pm10i$ which lie on the open left half plane. For $m=3$
and $r=0$, the zeros of 
\[
P_{r}(z)=z^{6}+1220z^{3}+1030301
\]
 do not lie on the given rays as seen in Figure \ref{fig:counterexp}.

\begin{figure}
\begin{centering}
\includegraphics[scale=0.5]{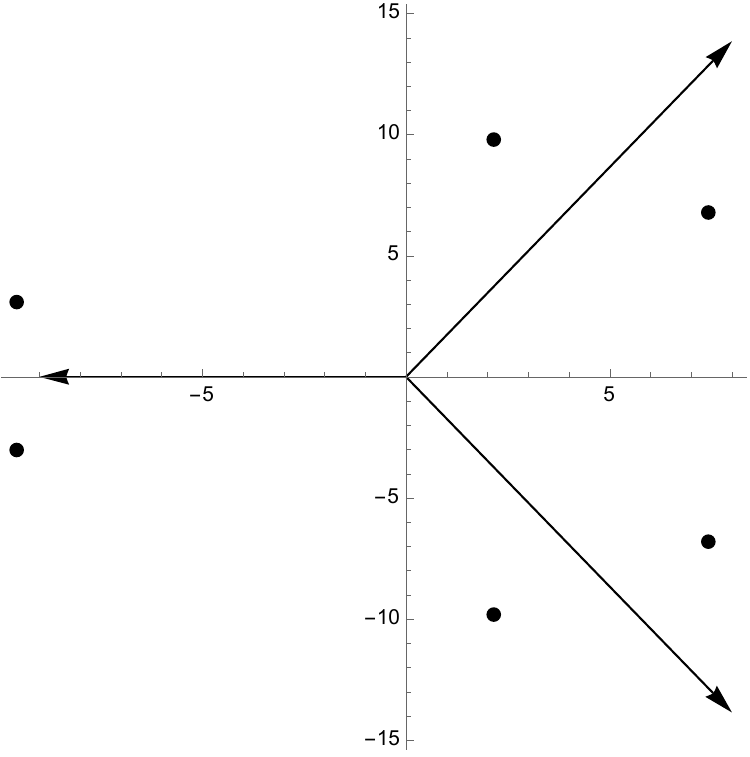}
\par\end{centering}
\caption{\label{fig:counterexp}Zeros of $z^{6}+1220z^{3}+1030301$}

\end{figure}

We state the last two theorems of the paper and provide an example
for each theorem.
\begin{thm}
\label{thm:zerossectorentire}Suppose $f(z)=\sum_{j=0}^{\infty}a_{j}z^{j}$,
$a_{j}\in\mathbb{R}$, is a real entire function whose zeros lie on
the plane sector $2\pi/3<\left|\Arg z\right|\le\pi$. If $f(z)$ has
the Weierstrass factorization 
\[
f(z)=z^{p}e^{Bz+C}\prod_{j=1}^{\infty}\left(1-\frac{z}{z_{j}}\right)e^{z/z_{j}}
\]
where $\sum_{j=1}^{\infty}1/|z_{j}|$ converges and $B+\sum_{j=1}^{\infty}\frac{1}{z_{j}}>0$,
then for any $0\le r<3$, the zeros of the entire function 
\[
f_{r}(z)=\sum_{j\equiv r\pmod3}a_{j}z^{j}
\]
lie on the $3$ radial rays: $\Im z^{3}=0$ and $\Re z^{3}\le0$. 
\end{thm}

For an example, we consider 
\[
f(z)=e^{2z}(z+1+i)(z+1-i)
\]
and note that the zeros $-1\pm i$ of $f(z)$ lie on the sector $2\pi/3<\left|\Arg z\right|\le\pi$
and 
\[
2+\frac{1}{-1+i}+\frac{1}{-1-i}=1>0.
\]
The function $f_{r}(z)$ for $m=3$ and $r=0$ has infinitely many
zeros, some of them are plotted in Figure \ref{fig:entirem3}.

\begin{figure}
\begin{centering}
\includegraphics[scale=0.5]{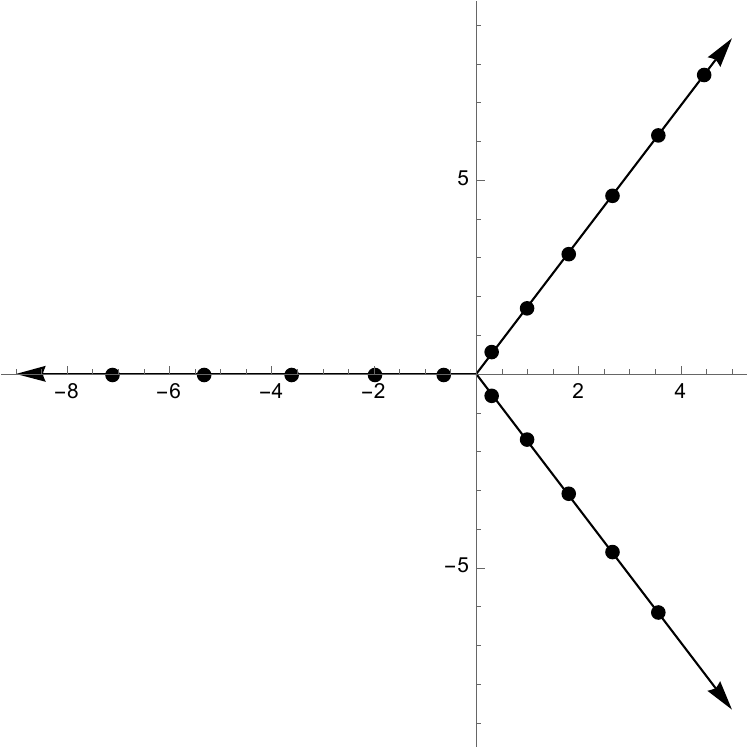}
\par\end{centering}
\caption{\label{fig:entirem3}Zeros of $f_{r}(z)$ for $f(z)=e^{2z}(z+1+i)(z+1-i)$}

\end{figure}

\begin{thm}
\label{thm:zeroshalfplaneentire}Suppose $f(z)=\sum_{j=0}^{\infty}a_{j}z^{j}$,
$a_{j}\in\mathbb{R}$, is a real entire function whose zeros lie on
the left half plane $\Re z<0$. If $f(z)$ has the Weierstrass factorization
\[
f(z)=z^{p}e^{Az^{2}+Bz+C}\prod_{j=1}^{\infty}\left(1-\frac{z}{z_{j}}\right)e^{z/z_{j}}
\]
where $A\ge0$, $\sum_{j=1}^{\infty}1/|z_{j}|$ converges, and $B+\sum_{j=1}^{\infty}\frac{1}{z_{j}}>0$,
then for any $0\le r<4$, the zeros of the entire function 
\[
f_{r}(z)=\sum_{j\equiv r\pmod4}a_{j}z^{j}
\]
lie on the $4$ radial rays: $\Im z^{4}=0$ and $\Re z^{4}\le0$. 
\end{thm}

For the plot of some zeros of $f_{r}(z)$ when $m=4$, $r=0$, and
\[
f_{r}(z)=e^{z^{2}+2z}(z+1+i)(z+1-i),
\]
see Figure \ref{fig:entirem4}.

\begin{figure}
\begin{centering}
\includegraphics[scale=0.5]{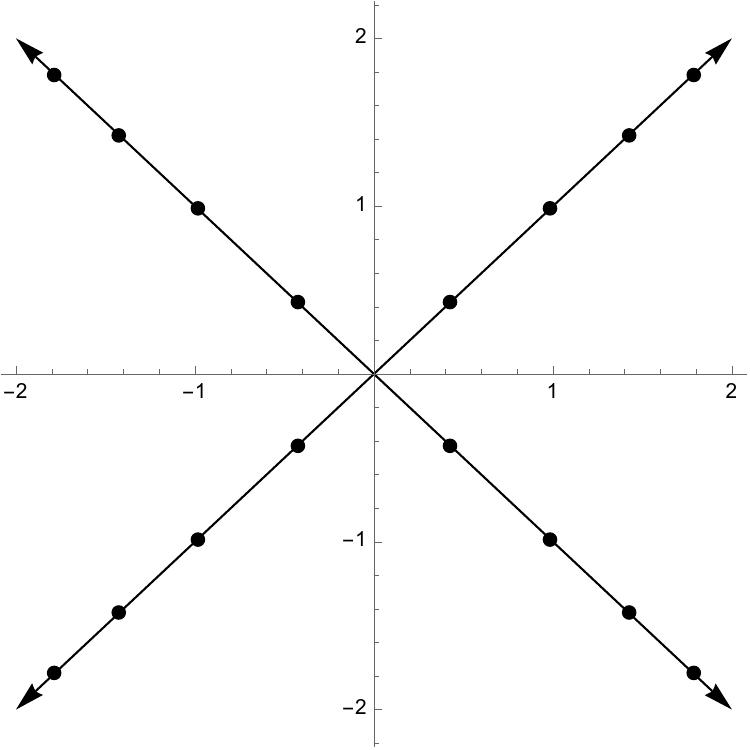}
\par\end{centering}
\caption{\label{fig:entirem4}Zeros of $f_{r}(z)$ for $f(z)=e^{z^{2}+2z}(z+1+i)(z+1-i)$}

\end{figure}

\section{Polynomials with arithmetic exponents}

In this section, we will prove Theorems \ref{thm:zerossectorpoly}
and \ref{thm:zerohalfplanepoly}. We first note that it suffices to
prove these theorems when the lead coefficient of $P(z)$ is $1$.
For any $m\in\mathbb{N}$ and $0\le r<m$, if we let
\[
P_{r}(z)=\sum_{\substack{j\equiv r\pmod m\\
0\le j\le n
}
}a_{j}z^{j}=\sum_{j=0}^{\left\lfloor n/m\right\rfloor }a_{mj+r}z^{mj+r},
\]
then
\begin{equation}
e^{-r\pi i/m}P_{r}(e^{\pi i/m}z)=\sum_{j=0}^{\left\lfloor n/m\right\rfloor }(-1)^{j}a_{jm+r}z^{jm+r}.\label{eq:P_rrotation}
\end{equation}
Let $H(z)$ be the polynomial in the right side of this equation and
$\omega_{k}=e^{(2k-1)\pi i/m}$, $0\le k<m$, be the $m$-th roots
of $-1$. We note from 
\[
P(z)=\sum_{j=0}^{n}a_{j}z^{j}
\]
that 
\begin{equation}
\sum_{k=0}^{m-1}\omega_{k}^{-r}P(\omega_{k}z)=\sum_{k=0}^{m-1}\sum_{j=0}^{n}\omega_{k}^{-r+j}a_{j}z^{j}.\label{eq:sumrootsneg1}
\end{equation}
Since $\omega_{k}$, $0\le k<m$, are the $m$-th roots of $-1$,
we have
\[
\sum_{k=0}^{m-1}\omega_{k}^{j-r}=\begin{cases}
0 & \mbox{if }j\not\equiv r\mbox{ mod }m\\
m(-1)^{(j-r)/m} & \mbox{if }j\equiv r\mbox{ mod }m
\end{cases}.
\]
We conclude from this equation, (\ref{eq:P_rrotation}), and (\ref{eq:sumrootsneg1})
that
\[
\sum_{k=0}^{m-1}\omega_{k}^{-r}P(\omega_{k}z)=m\sum_{\substack{j\equiv r\pmod m\\
0\le j\le n
}
}(-1)^{(j-r)/m}a_{j}z^{j}=mH(z).
\]
Since $z_{j}$, $1\le j\le n$, are the zeros of the monic polynomial
$P(z)$, we have 
\[
P(z)=\prod_{j=1}^{n}(z-z_{j})
\]
and consequently
\begin{equation}
mH(z)=\sum_{k=0}^{m-1}\omega_{k}^{-r}P(\omega_{k}z)=\sum_{k=0}^{m-1}\omega_{k}^{-r}\prod_{j=1}^{n}(\omega_{k}z-z_{j}).\label{eq:Hformula}
\end{equation}

To prove Theorems \ref{thm:zerossectorpoly} and \ref{thm:zerohalfplanepoly},
we will show that the zeros of 
\[
H(z)=e^{-r\pi i/m}P_{r}(e^{\pi i/m}z)
\]
lie on the $m$ radial rays $\Im z^{m}=0$ and $\Re z^{m}\ge0$. We
note that the positive real ray is one of them. Our main strategy
is to show that $H(z)$ has $n$ zeros on these rays and the claim
will follow from the fundamental theorem of algebra. From (\ref{eq:P_rrotation}),
we conclude that $H(z)$ has a zero at $z=0$ or order $r$ and if
$z>0$ is a positive real zero of $H(z)$, then so are $e^{k\pi i/m}z$,
$0\le k<m$. Thus it suffices to show that $H(z)$ has $\left\lfloor n/m\right\rfloor $
positive real zeros. 

For the reason mentioned above, we consider the polynomial $H(z)$
when $z$ is a positive real number. For each $1\le j\le n$ and $0\le k<m$,
and $z\in\mathbb{R}^{+}$, we let 
\begin{equation}
q_{j,k}=\frac{\omega_{k}z-z_{j}}{\omega_{0}z-\overline{z_{j}}}\label{eq:qjkdef}
\end{equation}
and 
\begin{equation}
\theta_{j}=\Arg(\omega_{1}z-z_{j})\label{eq:thetajdef}
\end{equation}
where $\omega_{k}=e^{(2k-1)\pi i/m}$. Since 
\[
\Re(\omega_{1}z-z_{j})>0,
\]
we have 
\begin{equation}
-\pi/2<\theta_{j}<\pi/2.\label{eq:thetajrange}
\end{equation}
We normalize $mH(z)$ by dividing the right side of (\ref{eq:Hformula})
by $\prod_{j=1}^{n}(\omega_{0}z-\overline{z_{j}})e^{i\theta_{j}}$
and conclude this quotient is $\sum_{k=0}^{m-1}\omega_{k}^{-r}\prod_{j=1}^{n}q_{j,k}e^{-i\theta_{j}}$
or equivalently
\begin{equation}
g(\theta):=\sum_{k=0}^{m-1}\omega_{k}^{-r}e^{-in\theta}\prod_{j=1}^{n}q_{j,k}\label{eq:gtheta}
\end{equation}
where $\theta:=\sum_{j=1}^{n}\theta_{j}$. The two lemmas below show
the first two terms (when $k=0$ and $k=1$) in (\ref{eq:gtheta})
dominate the other terms in the sum. 
\begin{lem}
\label{lem:modulusq-sector}Suppose $z\in\mathbb{R}^{+}$ and $q_{j,k}$,
$1\le j\le n$, $0\le k<m$, are defined as in (\ref{eq:qjkdef}).
For any $m\ge2$, if $\pi-\pi/m<\left|\Arg z_{j}\right|\le\pi$ for
all $1\le j\le n$, then for any $0\le k<m$ we have 
\begin{equation}
\prod_{j=1}^{n}|q_{j,k}|\le1.\label{eq:prodineq}
\end{equation}
Moreover the equation occurs if and only if $k=0$ or $k=1$. 
\end{lem}

\begin{proof}
In the case $k=1$, we have
\begin{equation}
\prod_{j=1}^{n}q_{j,1}=\prod_{j=1}^{n}\frac{\omega_{1}z-z_{j}}{\omega_{0}z-\overline{z_{j}}}=\prod_{j=1}^{n}e^{2i\theta_{j}}=e^{2i\theta}\label{eq:prodq1}
\end{equation}
where the second equation comes from (\ref{eq:thetajdef}) and the
fact that $\omega_{1}z-z_{j}$ and $\omega_{0}z-\overline{z_{j}}$
are two complex conjugates. Thus 
\[
\prod_{j=1}^{n}|q_{j,1}|=1.
\]
If $k=0$ and $z_{j}\in\mathbb{R}$ then $q_{j,0}=1$. On the other
hand, if $k=0$ and $z_{j}\notin\mathbb{R}$, then both $z_{j}$ and
$\overline{z}_{j}$ are zeros of $P(z)$. As a consequence, the $\prod_{j=1}^{n}q_{j,0}$
contains the factor 
\[
\frac{\omega_{0}z-z_{j}}{\omega_{0}z-\overline{z_{j}}}\cdot\frac{\omega_{0}z-\overline{z_{j}}}{\omega_{0}z-z_{j}}=1
\]
from which we deduce that
\begin{equation}
\prod_{j=1}^{n}q_{j,0}=1.\label{eq:prodq0}
\end{equation}

To complete the lemma, it remains to show 
\[
\prod_{j=1}^{n}|q_{j,k}|<1
\]
for $2\le k<m$ and $m>2$. If $z_{j}=r_{j}e^{i\phi_{j}}$, $\pi-\pi/m<|\phi_{j}|\le\pi$,
then for $2\le k<m$, 
\begin{eqnarray}
|\omega_{k}z-z_{j}|^{2} & = & \left(z\cos\frac{(2k-1)\pi}{m}-r_{j}\cos\phi_{j}\right)^{2}+\left(z\sin\frac{(2k-1)\pi}{m}-r_{j}\sin\phi_{j}\right)^{2}\nonumber \\
 & = & z^{2}+r_{j}^{2}-2zr_{j}\cos\left(\phi_{j}-\frac{(2k-1)\pi}{m}\right)\nonumber \\
 & = & z^{2}+r_{j}^{2}-2zr_{j}\cos\left(\phi_{j}^{*}-\frac{(2k-1)\pi}{m}\right)\label{eq:absnum}
\end{eqnarray}
where 
\[
\phi_{j}^{*}=\begin{cases}
\phi_{j} & \text{ if }\phi_{j}\in(\pi/2,\pi]\\
\phi_{j}+2\pi & \text{ if }\phi_{j}\in(-\pi,-\pi/2)
\end{cases}.
\]
Similarly
\begin{equation}
|\omega_{0}z-\overline{z_{j}}|^{2}=z^{2}+r_{j}^{2}-2zr_{j}\cos\left(\phi_{j}^{*}-\frac{\pi}{m}\right).\label{eq:absdenom}
\end{equation}
From the assumption $\pi-\pi/m<|\phi_{j}|\le\pi$, one has $\pi-\pi/m<\phi_{j}^{*}<\pi+\pi/m$
and consequently
\[
\frac{m-2k+2}{m}\pi>\phi_{j}^{*}-\frac{(2k-1)\pi}{m}>\frac{m-2k}{m}\pi,\qquad2\le k<m,
\]
and 
\[
\pi>\phi_{j}^{*}-\frac{\pi}{m}>\frac{m-2}{m}\pi.
\]
Thus 
\[
\cos\left(\phi_{j}^{*}-\frac{(2k-1)\pi}{m}\right)>\cos\left(\phi_{j}^{*}-\frac{\pi}{m}\right).
\]
We compare (\ref{eq:absnum}) and (\ref{eq:absdenom}) and deduce
that $|q_{j,k}|=|\omega_{k}z-z_{j}|/|\omega_{0}z-\overline{z_{j}}|<1$
for all $1\le j\le n$ and $2\le k<m$.
\end{proof}
In the case $m$ is even, we can expand the sector containing $z_{j}$,
$1\le j\le n$, in the condition of Lemma \ref{lem:modulusq-sector}
to a half plane in the lemma below.
\begin{lem}
\label{lem:modulusq-halfplane}Suppose $z\in\mathbb{R}^{+}$ and $q_{j,k}$,
$1\le j\le n$, $0\le k<m$, are defined as in (\ref{eq:qjkdef}).
If $\Re z_{j}<0$ for all $1\le j\le n$ and $m\ge2$ is even, then
for any $0\le k<m$ we have 
\[
\prod_{j=1}^{n}|q_{j,k}|\le1.
\]
Moreover the equation occurs if and only if $k=0$ or $k=1$. 
\end{lem}

\begin{proof}
For the case $k=0$ and $k=1$ we have (\ref{eq:prodq0}) and (\ref{eq:prodq1}).
We now consider $2\le k<m$. We note that if $\pi-\pi/m<|\arg z_{j}|\le\pi$,
then by the proof of Lemma \ref{lem:modulusq-sector}, $|q_{j,k}|<1$
for $2\le k<m$. It remains to show that 
\[
\prod_{j}|q_{j,k}|<1
\]
where the product runs over values of $1\le j\le n$ such that $\pi/2<|\arg z_{j}|<\pi-\pi/m$.
Since $z_{j}$ and $\overline{z_{j}}$ are zeros of $P(z)$, if the
modulus of
\[
q_{j,k}=\frac{\omega_{k}z-z_{j}}{\omega_{0}z-\overline{z_{j}}}
\]
is a term in this product, then so is the modulus of
\[
\frac{\omega_{k}z-\overline{z_{j}}}{\omega_{0}z-z_{j}}.
\]
To complete the proof of this lemma, we will show that if $\pi/2<\arg z_{j}<\pi-\pi/m$,
then for all $2\le k<m$
\[
\left|\frac{\omega_{k}z-z_{j}}{\omega_{0}z-\overline{z_{j}}}\frac{\omega_{k}z-\overline{z_{j}}}{\omega_{0}z-z_{j}}\right|<1
\]
or equivalently 
\[
\left|\omega_{k}z-z_{j}\right|\left|\omega_{k}z-\overline{z_{j}}\right|<\left|\omega_{0}z-z_{j}\right|\left|\omega_{0}z-\overline{z_{j}}\right|.
\]
From the property of complex conjugate that 
\[
\left|\omega_{k}z-z_{j}\right|\left|\omega_{k}z-\overline{z_{j}}\right|=\left|\overline{\omega_{k}}z-\overline{z_{j}}\right|\left|\overline{\omega_{k}}z-z_{j}\right|,\qquad0\le k<m,
\]
and $\overline{\omega_{0}}=\omega_{1}$, it suffices to prove that
\begin{equation}
\left|\omega_{k}z-z_{j}\right|\left|\omega_{k}z-\overline{z_{j}}\right|<\left|\omega_{1}z-z_{j}\right|\left|\omega_{1}z-\overline{z_{j}}\right|\label{eq:proddistineq}
\end{equation}
for any $2\le k<m$ such that $\omega_{k}=e^{(2k-1)\pi i/m}$ lies
on the open upper half plane $(\omega_{k}\ne-1$ since $m$ is even).
We note that the angles of those $\omega_{k}$ on the upper half plane
and $\omega_{1}$ belong to $[\pi/m,\pi-\pi/m]$. Thus to prove \ref{eq:proddistineq},
we will show that the function 
\[
f(\theta)=\left|ze^{i\theta}-z_{j}\right|^{2}\left|ze^{i\theta}-\overline{z_{j}}\right|^{2}
\]
attains is absolute maximum on $[\pi/m,\pi-\pi/m]$ only at $\theta=\pi/m$.
We note that 
\begin{align}
f(\theta) & =\left|z^{2}e^{2i\theta}-2ze^{i\theta}\Re z_{j}+|z_{j}|^{2}\right|^{2}\nonumber \\
 & =z^{4}-\left(4z^{3}+4z|z_{j}|^{2}\right)\Re z_{j}\cos\theta+4z^{2}(\Re z_{j})^{2}+|z_{j}|^{4}+2z^{2}|z_{j}|^{2}\cos2\theta\label{eq:ftheta}
\end{align}
from which we deduce 
\[
f'(\theta)=\left(4z^{3}+4z|z_{j}|^{2}\right)\Re z_{j}\sin\theta-4z^{2}|z_{j}|^{2}\sin\theta\cos\theta.
\]
Hence $f(\theta)$ has at most one critical point on the interval
$[\pi/m,\pi-\pi/m]$. We also note that for $\theta\in[\pi/m,\pi/2]$,
$f'(\theta)<0$ since $\Re z_{j}<0$, $z>0$, and $\cos\theta>0$.
Thus $f(\theta)$ will not attain its maximum value at such critical
point (if exists). We conclude that the maximum value of $f(\theta)$
occur at either $\theta=\pi/m$ or $\theta=\pi-\pi/m$. Since 
\[
\cos2(\pi-\pi/m)=\cos(2\pi/m),
\]
we deduce from (\ref{eq:ftheta}) and $\Re z_{j}<0$ that $f(\pi/m)>f(\pi-\pi/m)$.
Hence $f(\theta)$ attains it maximum value on $[\pi/m,\pi-\pi/m]$
only at $\pi/m$ and the lemma follows. 
\end{proof}
We recall that we want to show $H(z)$ has $\left\lfloor n/m\right\rfloor $
positive real zeros. Since $g(\theta)$ defined as in (\ref{eq:gtheta})
is the quotient of $mH(z)$ and $\prod_{j=1}^{n}(\omega_{0}z-\overline{z_{j}})e^{i\theta_{j}}$,
it suffices to show (i) $\theta=\theta(z)$ bijectively map the interval
$z\in(0,\infty)$ to $(0,\pi/m)$ and (ii) $g(\theta)$ has $\left\lfloor n/m\right\rfloor $
zeros on $\theta\in(0,\pi/m)$. The first claim is proved in the lemma
below.
\begin{lem}
\label{lem:monotone}For each $z\in\mathbb{R}^{+}$, and each $1\le j\le n$,
$0\le k<m$, let $q_{j,k}$ and $\theta_{j}$ be defined as in (\ref{eq:qjkdef})
and (\ref{eq:thetajdef}) . If for any $1\le j\le n$, $z_{j}=r_{j}e^{i\phi_{j}}$,
$\pi/2<|\phi_{j}|\le\pi$, lie on the open left half plane then 
\[
\theta:=\frac{1}{n}\sum_{j=1}^{n}\theta_{j}
\]
is an increasing function in $z\in(0,\infty)$ and this function maps
$(0,\infty)$ onto $(0,\pi/m)$. 
\end{lem}

\begin{proof}
We note that 
\begin{align}
n\frac{d\theta}{dz} & =\sum_{j=1}^{n}\frac{d\theta_{j}}{dz}\nonumber \\
 & =\sum_{j=1}^{n}\frac{d(\tan\theta_{j})}{dz}\left(\frac{d(\tan\theta_{j})}{d\theta_{j}}\right)^{-1}\nonumber \\
 & =\sum_{j=1}^{n}\frac{d(\tan\theta_{j})}{dz}\cos^{2}\theta_{j}.\label{eq:sumderiv}
\end{align}
where from (\ref{eq:thetajdef})
\begin{align}
\frac{d(\tan\theta_{j})}{dz} & =\frac{d}{dz}\left(\frac{\Im(\omega_{1}z-z_{j})}{\Re(\omega_{1}z-z_{j})}\right).\nonumber \\
 & =\frac{d}{dz}\left(\frac{z\sin(\pi/m)-r_{j}\sin\phi_{j}}{z\cos(\pi/m)-r_{j}\cos\phi_{j}}\right)\nonumber \\
 & =\frac{r_{j}\left(\cos(\pi/m)\sin\phi_{j}-\sin(\pi/m)\cos\phi_{j}\right)}{(z\cos(\pi/m)-r_{j}\cos\phi_{j})^{2}}\nonumber \\
 & =\frac{r_{j}\sin(\phi_{j}-\pi/m)}{(z\cos(\pi/m)-r_{j}\cos\phi_{j})^{2}}.\label{eq:derivetheta}
\end{align}
We recall that all $z_{j}$, $1\le j\le n$, lie on the open left
half plane and they are the zeros of the real polynomial $P(z)$.
If $z_{j}\in\mathbb{R^{-}}$, then $\phi_{j}=\pi$ and the equation
above implies $d(\tan\theta_{j})/dz\ge0$. On the other hand, if $z_{j}=r_{j}e^{i\phi_{j}}\notin\mathbb{R}^{-}$
is a zero of $P(z)$, then so is $\overline{z_{j}}$. Let $k\ne j$
be the index of such zero, i.e., $z_{k}=\overline{z_{j}}$ . Since
both $z_{j}$ and $z_{k}$ lie on the open left half plane, without
loss of generality, we assume $\phi_{k}\in(\pi/2,\pi)$ and $\phi_{j}\in(-\pi,-\pi/2)$.
The summation in (\ref{eq:sumderiv}) contains
\begin{align}
 & \frac{d(\tan\theta_{j})}{dz}\cos^{2}\theta_{j}+\frac{d(\tan\theta_{k})}{dz}\cos^{2}\theta_{k}\nonumber \\
= & \left(\frac{d(\tan\theta_{j})}{dz}+\frac{d(\tan\theta_{k})}{dz}\right)\cos^{2}\theta_{j}+\frac{d(\tan\theta_{k})}{dz}(\cos^{2}\theta_{k}-\cos^{2}\theta_{j}).\label{eq:sumtwoderiv}
\end{align}
To prove this sum is nonnegative we note from $z_{k}=\overline{z_{j}}$
that $r_{j}=r_{k}$ and $\phi_{j}+\phi_{k}=0$. Thus (\ref{eq:derivetheta})
yields
\[
\frac{d(\tan\theta_{j})}{dz}+\frac{d(\tan\theta_{k})}{dz}=\frac{2r_{k}\sin(-\pi/m)\cos(\phi_{k})}{(z\cos(\pi/m)-r_{k}\cos\phi_{k})^{2}}>0.
\]

Since $\phi_{k}\in(\pi/2,\pi)$, (\ref{eq:derivetheta}) implies that
$d(\tan\theta_{k})/dz>0$. Consequently to prove (\ref{eq:sumtwoderiv})
is positive, it suffices to show $\cos^{2}\theta_{k}-\cos^{2}\theta_{j}>0$.
Indeed, we have 
\[
\tan\theta_{j}=\frac{z\sin(\pi/m)-r_{j}\sin\phi_{j}}{z\cos(\pi/m)-r_{j}\cos\phi_{j}}>\frac{z\sin(\pi/m)-r_{j}\sin\phi_{k}}{z\cos(\pi/m)-r_{j}\cos\phi_{k}}=\tan\theta_{k}
\]
since $\sin\phi_{k}=-\sin\phi_{j}>0$ and $\cos\phi_{j}=\cos\phi_{k}<0$.
Hence 
\[
\cos^{2}\theta_{k}=\frac{1}{1+\tan^{2}\theta_{k}}>\frac{1}{1+\tan^{2}\theta_{j}}=\cos^{2}\theta_{j}.
\]

We have shown that (\ref{eq:sumderiv}) is positive and thus $\theta$
is an increasing function in $z\in(0,\infty)$. To finish the proof
of this lemma, it remains to show $\lim_{z\rightarrow0}\theta=0$
and $\lim_{z\rightarrow\infty}\theta=\pi/m$. From 
\[
\tan\theta_{j}=\frac{z\sin(\pi/m)-r_{j}\sin\phi_{j}}{z\cos(\pi/m)-r_{j}\cos\phi_{j}},
\]
we conclude
\begin{align*}
\lim_{z\rightarrow0}\tan\theta_{j} & =\tan\phi_{j},\\
\lim_{z\rightarrow\infty}\tan\theta_{j} & =\tan\frac{\pi}{m}.
\end{align*}
We combine these two limits with the fact that $-\pi/2<\theta_{j}<\pi/2$
(given in (\ref{eq:thetajrange})) and $\pi/2<|\phi_{j}|\le\pi$ to
obtain 
\begin{equation}
\lim_{z\rightarrow0}\theta_{j}=\begin{cases}
\phi_{j}-\pi & \text{ if }\pi/2<\phi_{j}\le\pi\\
\phi_{j}+\pi & \text{ if }-\pi<\phi_{j}<-\pi/2
\end{cases}\label{eq:limthetaj0}
\end{equation}
and
\[
\lim_{z\rightarrow\infty}\theta_{j}=\frac{\pi}{m}.
\]
The definition $\theta=\sum_{j=1}^{n}\theta_{j}$ yields $\lim_{z\rightarrow\infty}\theta=\pi/m$.
Finally, this definition, (\ref{eq:limthetaj0}), and the fact that
both $r_{j}e^{i\phi_{j}}$ and $r_{j}e^{-i\phi_{j}}$ are zeros of
$P(z)$ give $\lim_{z\rightarrow0}\theta=0$. 
\end{proof}
Having proved that $\theta=\theta(z)$ bijectively map the interval
$z\in(0,\infty)$ to $(0,\pi/m)$, we now turn our attention to proving
that $g(\theta)$, defined as in (\ref{eq:gtheta}) has $\left\lfloor n/m\right\rfloor $
zeros on the interval $(0,\pi/m)$. We recall that $g(\theta)$ is
the quotient of $mH(z)$ and $\prod_{j=1}^{n}(\omega_{0}z-\overline{z_{j}})e^{i\theta_{j}}$.
From (\ref{eq:thetajdef}) and the fact that $\omega_{1}z-z_{j}$
and $\omega_{0}z-\overline{z_{j}}$ are complex conjugate, we conclude
that $g(\theta)\in\mathbb{R}$ for $\theta\in(0,\pi/m)$. From (\ref{eq:prodq0})
and (\ref{eq:prodq1}), the sum of the first two terms of 
\[
g(\theta)=\sum_{k=0}^{m-1}\omega_{k}^{-r}e^{-in\theta}\prod_{j=1}^{n}q_{j,k}
\]
when $k=0$ and $k=1$ is 
\[
e^{-in\theta-ri\pi/m}+e^{in\theta+ri\pi/m}=2\cos\left(n\theta+\frac{r\pi}{m}\right).
\]
From Lemmas \ref{lem:modulusq-sector} and \ref{lem:modulusq-halfplane},
we obtain the following upper bound for the sum of the remainder terms
when $2\le k<m$: 
\[
\left|\sum_{k=2}^{m-1}\omega_{k}^{-r}e^{-in\theta}\prod_{j=1}^{n}q_{j,k}\right|\le\sum_{k=2}^{m-1}\prod_{j=1}^{n}|q_{j,k}|<m-2.
\]
When $m\le4$, the right side is less than $2$ and thus at the values
of $\theta$ where $\cos\left(n\theta+\frac{r\pi}{m}\right)=\pm1$,
$g(\theta)$ is positive/negative respectively. Specifically, when
\begin{equation}
\theta_{h}=\frac{h-r/m}{n}\pi,\qquad1\le h\le\left\lfloor n/m\right\rfloor ,\label{eq:thetah}
\end{equation}
the sign of $g(\theta_{h})$ is $(-1)^{h}$. Moreover $\theta_{h}>0$
and 
\[
\theta_{h}\le\frac{n/m-r/m}{n}\pi\le\frac{\pi}{m}.
\]
In the case $\theta_{h}=\pi/m$, we replace such value of $\theta_{h}$
by $(\pi/m)^{-}$ so that all $\theta_{h}$ defined in (\ref{eq:thetah})
lie on the interval $(0,\pi/m)$. By the intermediate value theorem,
$g(\theta)$ has at least one zero on each interval $(\theta_{h},\theta_{h+1})$,
$1\le h<\left\lfloor n/m\right\rfloor $. Thus we obtain at least
$\left\lfloor n/m\right\rfloor -1$ zeros of $g(\theta)$ on $(0,\pi/m)$,
each of which produces one positive real zero of $H(z)$ by Lemma
\ref{lem:monotone}. 

We claim that $H(z)$ has exactly $\left\lfloor n/m\right\rfloor $
positive real zeros. Indeed, if $z_{j}\notin\mathbb{R}$, we pair
the two terms $z-z_{j}$ and $z-\overline{z_{j}}$ on the right side
of
\[
P(z)=\prod_{j=1}^{n}(z-z_{j})
\]
and note that 
\[
(z-z_{j})(z-\overline{z_{j}})=z^{2}-2\Re z_{j}+|z_{j}|^{2}.
\]
Since all the zeros of $P(z)$ lie on the left half plane, $P(z)$
is a product of polynomials with positive coefficients and consequently
all the coefficients of $P(z)$ are positive. By Descartes' rule of
signs, the number of positive zeros of 
\[
H(z)=\sum_{j=0}^{\left\lfloor n/m\right\rfloor }(-1)^{j}a_{jm+r}z^{jm+r}
\]
is either $\left\lfloor n/m\right\rfloor $ or less than that by an
even number. As we have shown that $H(z)$ has at least $\left\lfloor n/m\right\rfloor -1$
positive real zeros, the number of such zeros must be $\left\lfloor n/m\right\rfloor $.
This concludes the proofs of Theorems \ref{thm:zerossectorpoly} and
\ref{thm:zerohalfplanepoly}. 

\section{Entire functions with arithmetic exponents}

In this section, we will prove Theorems \ref{thm:zerossectorentire}
and \ref{thm:zeroshalfplaneentire}. Suppose $m=3$ or $4$ and $0\le r<m$.
From an entire function $f(z)=\sum_{j=0}^{\infty}a_{j}z^{j}$, we
define 
\[
f_{r}(z)=\sum_{j=0}^{\infty}a_{jm+r}z^{jm+r}.
\]
 We assume that 
\[
f(z)=z^{p}e^{Az^{2}+Bz+C}\prod_{j=1}^{\infty}\left(1-\frac{z}{z_{j}}\right)e^{z/z_{j}}
\]
where $\sum_{j=1}^{\infty}1/z_{j}$ converges absolutely, $B+\sum_{j=1}^{\infty}\frac{1}{z_{j}}>0$,
and 
\begin{equation}
\begin{cases}
A\ge0\text{ and }\text{\ensuremath{\Re z_{j}<0}} & \text{ if }m=4\\
A=0\text{ and }2\pi/3<\left|\Arg z_{j}\right|\le\pi & \text{ if }m=3
\end{cases}.\label{eq:cond}
\end{equation}
Since $\sum_{j=1}^{\infty}1/z_{j}$ converges absolutely, we can collect
the factors $e^{z/z_{j}}$, $1\le j<\infty$ to have 
\[
f(z)=z^{p}e^{C}\exp\left(Az^{2}+\left(B+\sum_{j=1}^{\infty}\frac{1}{z_{j}}\right)z\right)\prod_{j=1}^{\infty}\left(1-\frac{z}{z_{j}}\right).
\]
Under the condition (\ref{eq:cond}), we will show the zeros of $f_{r}(z)$
lie on the $m$ radial rays: $\Im z^{m}=0$ and $\Re z^{m}\le0$.
With a similar argument as that in the previous section, it is equivalent
to show that the zeros of 
\begin{align*}
e^{-r\pi i/m}f_{r}(e^{\pi i/m}z) & =\sum_{j=0}^{\infty}(-1)^{j}a_{jm+r}z^{jm+r}\\
 & =\frac{1}{m}\sum_{k=0}^{m-1}\omega_{k}^{-r}f(\omega_{k}z)
\end{align*}
lie on the $m$ radial rays: $\Im z^{m}=0$ and $\Re z^{m}\ge0$.
If we denote the polynomial above by $h(z)$, then it suffices to
prove that the zeros of $h(z)$ lie on these rays on any compact set
$K$. On that compact set, the sequence of polynomials
\[
P_{n}(z)=z^{p}e^{C}\left(1+\frac{1}{n}\left(Az^{2}+\left(B+\sum_{j=1}^{\infty}\frac{1}{z_{j}}\right)z\right)\right)^{n}\prod_{j=1}^{n}\left(1-\frac{z}{z_{j}}\right)
\]
converges uniformly to $f(z)$ as $n\rightarrow\infty$. Consequently
the sequence of polynomials 
\begin{equation}
\frac{1}{m}\sum_{k=0}^{m-1}\omega_{k}^{-r}P_{n}(\omega_{k}z)\label{eq:Pnmodr}
\end{equation}
converges uniformly on $K$ to 
\begin{equation}
\frac{1}{m}\sum_{k=0}^{m-1}\omega_{k}^{-r}f(\omega_{k}z)\label{eq:fmodr}
\end{equation}
as $n\rightarrow\infty$. 

The zeros of $P_{n}(z)$ are $z_{j}$, $1\le j<\infty$, and those
of 
\begin{equation}
Az^{2}+\left(B+\sum_{j=1}^{\infty}\frac{1}{z_{j}}\right)z+n.\label{eq:quadratic}
\end{equation}
If $m=3$, then $A=0$ and all the zeros of $P_{n}(z)$ lie on the
sector $2\pi/3<|\Arg z|\le\pi$ since 
\[
B+\sum_{j=1}^{\infty}\frac{1}{z_{j}}>0.
\]
From the previous section, the zeros of the polynomial (\ref{eq:Pnmodr})
lie on the $m$ radial rays: $\Im z^{m}=0$ and $\Re z^{m}\ge0$.
Since this sequence converges to (\ref{eq:fmodr}) uniformly on $K$
as $n\rightarrow\infty$, the intersection of the set of zeros of
(\ref{eq:fmodr}) and $K$ must also be a subset of the same $m$
radial rays. Since $K$ is arbitrary, all the zeros of (\ref{eq:fmodr})
must lie on these rays and Theorem \ref{thm:zerossectorentire} follows.

On the other hand if $m=4$, then the zeros of (\ref{eq:quadratic})
lie on the left-half plane since the zeros are either 
\[
-n\left(B+\sum_{j=1}^{\infty}\frac{1}{z_{j}}\right)^{-1}<0
\]
in the case $A=0$ or a pair of complex conjugates whose real part
is 
\[
-\left(B+\sum_{j=1}^{\infty}\frac{1}{z_{j}}\right)\frac{1}{2A}<0
\]
in the case $A\ne0$ by the Vieta's formula. Theorem \ref{thm:zeroshalfplaneentire}
follows from the same arguments as those in the case $m=3$. We complete
the proof of Theorems \ref{thm:zerossectorentire} and \ref{thm:zeroshalfplaneentire}.


\begin{thebibliography}{1}
\bibitem[1]{bb}J. Borcea and P. Br\"and\'en, P\'olya-Schur master
theorems for circular domains and their boundaries, Annals of Math.,
170 (2009), 465-492.

\bibitem[2]{cf}G. Csordas, T. Forgács, Multiplier sequences, classes
of generalized Bessel functions and open problems, J. Math. Anal.
Appl. 443(2) (2016), pp. 631-651.

\bibitem[3]{ps}G. P\'olya and J. Schur, Uber zwei Arten von Faktorenfolgen
in der Theorie der algebraischen Gleichungen, J. Reine Angew. Math.,
144 (1914), 89-113.

\bibitem[4]{rs}Q.I. Rahman and G. Schmeisser, “Analytic Theory of
Polynomials,” London Math. Soc. Monographs (N.S.) 26, Oxford U.P.,
New York NY, 2002

\bibitem[5]{wang}X. Wang, A simple proof of Descartes’s rule of signs,
The American Mathematical Monthly. 111, (2004), No. 6, 525-526.
\end{thebibliography}
\end{document}